\documentclass{article}

\usepackage[letterpaper, margin=1.5in]{geometry}

\usepackage{lipsum}
\usepackage{amsfonts}
\usepackage{graphicx}
\usepackage{epstopdf}
\usepackage{algorithmic}
\ifpdf
  \DeclareGraphicsExtensions{.eps,.pdf,.png,.jpg}
\else
  \DeclareGraphicsExtensions{.eps}
\fi
\usepackage{times}                 
\usepackage{amsmath, amsthm}
\usepackage{mathtools}
\usepackage{nccmath}
\usepackage{amssymb,color, graphicx}
\usepackage{multirow}
\usepackage{epsfig}
\usepackage{algorithmic}
\usepackage{algorithm}
\usepackage{pgfplots}
\usepackage{tikz}
\usepackage[inline]{enumitem}
\usepackage{hyperref}
\usepackage{slashbox}
\usetikzlibrary{datavisualization}
\usetikzlibrary{pgfplots.groupplots}
\usetikzlibrary{external}
\tikzset{external/mode=graphics if exists}
\tikzset{
    png export/.style={
        external/system call/.add={}%
        {; convert -density 300 -transparent white "\image.pdf" "\image.png"}
    }
}
\tikzset{external/system call={lualatex
     \tikzexternalcheckshellescape -halt-on-error
     -interaction=batchmode -jobname "\image" "\texsource"}}

\usepgfplotslibrary{external}
\usepgfplotslibrary{dateplot}

\newcommand{\E}{\mathbb{E}}

\newcommand{\X}{\mathcal{X}}

\newcommand{\N}{\mathbb{N}}

\newcommand{\dprime}{{\prime\prime}}

\newtheorem{theorem}{Theorem}
\newtheorem{lemma}[theorem]{Lemma}

\newtheorem{definition}[theorem]{Definition}
\newtheorem{remark}{Remark}

\newcommand{\F}{\mathcal{F}}

\newcommand{\Hc}{\mathcal{H}}

\newcommand{\Ic}{\mathcal{I}}

\newcommand{\Xc}{\mathcal{X}}

\newcommand{\R}{\mathbb{R}}

\renewcommand{\Delta}{\varDelta}
\renewcommand{\Pi}{\varPi}
\newcommand{\Pb}{\mathbb{P}}

\DeclareMathOperator*{\argmin}{arg\,min}

\definecolor{pinegreen}{rgb}{0.1,0.5,0.2}

\usepackage{endnotes}
\let\footnote=\endnote
\let\enotesize=\normalsize
\def\notesname{Endnotes}%
\def\makeenmark{$^{\theenmark}$}
\def\enoteformat{\rightskip0pt\leftskip0pt\parindent=1.75em
  \leavevmode\llap{\theenmark.\enskip}}

\usepackage{natbib}
 \bibpunct[, ]{[}{]}{,}{a}{}{,}%
 \def\bibfont{\small}%
 \def\bibsep{\smallskipamount}%
 \def\bibhang{24pt}%
\newcommand{\TheTitle}{ Regularized Decomposition of High--Dimensional Multistage Stochastic Programs with Markov Uncertainty}

\title{{\TheTitle}}

\author{
  Tsvetan Asamov \thanks{Department of Operations Research and Financial Engieering, Princeton University} 
  \and
  Warren B. Powell\thanks{ Department of Operations Research and Financial Engieering, Princeton University}
}

\usepackage{amsopn}

\begin{document}

\maketitle

\begin{abstract}
We develop a quadratic regularization approach for the solution of high--dimensional multistage stochastic optimization problems characterized by a potentially large number of time periods/stages (e.g. hundreds), a high-dimensional resource state variable, and a Markov information process. The resulting algorithms are shown to converge to an optimal policy after a finite number of iterations under mild technical assumptions.  
Computational experiments are conducted using the setting of optimizing energy storage over a large transmission grid, which motivates both the spatial and temporal dimensions of our problem. 
Our numerical results indicate that the proposed methods exhibit significantly faster convergence than their classical counterparts, with greater gains observed for higher--dimensional problems.
\end{abstract}



\section{Introduction}
Multistage stochastic problems arise in a wide variety of real-world applications in fields as diverse as energy, finance, transportation and others.
In this paper we consider multistage stochastic linear programs that satisfy the following conditions:
\begin{enumerate*}[label=\itshape\roman*\upshape)]
\item the time horizon length $T$ is finite but potentially large (there may be hundreds of time periods and stages);
\item for each time period, the set of sample realizations of the exogenous information process is finite (and relatively small);
\item for each stage, the stage cost is a linear function of the decision. 
\end{enumerate*}

 Pereira and Pinto \cite{PereiraPinto} introduced a powerful algorithmic strategy known as Stochastic Dual Dynamic Programming (SDDP) that has received considerable attention for this problem class.
Despite its popularity, SDDP can exhibit slow convergence, especially in the setting of high--dimensional resource allocation problems.
A separate but important challenge arises when handling problems with long horizons which introduces algorithmic issues for both the setting of intertemporal independence, as well as when there is Markov dependence.  Not surprisingly, as practical problems grow in size, improving the rate of convergence of SDDP--type methods becomes an issue of growing importance.

Quadratic regularization has been among the most effective techniques for accelerating the convergence of scenario tree--based decomposition methods (see work by Ruszczy{\`n}ski \cite{ruszczynski1993regularized, ruszczynski1997decomposition, ruszczynski1997accelerating}).
However, its application to the SDDP framework has not been possible because of the exponential growth of the number of required incumbent solutions.
In this work, we propose a new regularization approach which overcomes that challenge.
The method can lead to much faster convergence by reducing the oscillation of solutions around distant vertices of the feasible regions where the addition of new cutting hyperplanes might be of little value.

This paper makes the following contributions:
\begin{enumerate*}[label=\itshape\roman*\upshape)]
\item We adopt notation that bridges the gap between dynamic programming and classical stochastic programming, which lays the foundation of our algorithmic strategy by identifying and clarifying the role of the post--decision information state;
\item We develop the first quadratic regularization method for the SDDP framework, with or without Markov dependence in the information process, that produces an optimal policy for a sampled model;
\item Unlike existing regularization methods on scenario trees, our approach remains computationally tractable even for problems that involve long time horizons;
\item Our numerical results indicate that the proposed approach exhibits faster convergence than classical SDDP
and is especially useful for problems with high--dimensional resource states. That makes the work relevant to a wide variety of practical applications.
\end{enumerate*}

Our numerical work uses the setting of optimizing energy over a fleet of storage devices for a congested transmission grid.  This problem class offers a realistic setting for testing the algorithm with anywhere from 50 to 500 batteries, allowing us to test the performance of the algorithm for resource state variables with widely varying dimensionality.  A separate challenge is that these problems exhibit a large number of time periods; our experiments model a day in 5--minute increments,
producing problems with 288 time periods.  
\section{Literature Review}
\label{s:lit}
The decomposition approach of Benders \cite{benders1962partitioning} and the L--shaped method of Van Slyke and Wets \cite{van1969shaped} originally focused on the solution of two--stage stochastic optimization problems.
Eventually, the idea was extended to the multi--period setting by Birge \cite{birge1985decomposition}, as well as Donohue and Birge \cite{donohue2006abridged} who considered successive Benders--type approximations of the recourse functions in the nested Benders decomposition algorithm for multistage problems on scenario trees.
Pereira and Pinto \cite{PereiraPinto} further extended the approach to develop Stochastic Dual Dynamic Programming which has become popular among practitioners.
On one hand, the method provides both lower and upper bounds, as well as clear convergence guarantees for many of its different
versions as has been discussed Shapiro \cite{shapiro2011analysis}, as well as Linowsky and Philpott \cite{linowsky2005convergence}.
Moreover, it is also very suitable for parallel computing and can be applied to problems with long time horizons.
Despite its progress towards overcoming the curse of dimensionality,
in its essence SDDP is a cutting plane method, a class of algorithms known to exhibit slow convergence (see \cite{ruszczynski2006nonlinear}), a behavior that is a byproduct of the well--known curse of dimensionality.
In general, their computational complexity grows exponentially with the dimension of the problem.
In the special case of only two time periods, the SDDP algorithm is equivalent to the well known cutting plane method of Kelley \cite{kelley1960cutting} which takes
$\displaystyle O\bigg(\frac{\ln\epsilon^{-1}}{2\ln 2}\Big[\frac{2}{\sqrt{3}}\Big]^{n-1}\bigg)$
iterations to achieve an $\epsilon$--optimal solution on an $n$--dimensional problem as pointed out by Nesterov and Nesterov \cite{nesterov2004introductory}.

Rockafellar \cite{rockafellar1976monotone} introduced the proximal point algorithm for the minimization of (deterministic) lower semicontinuous proper convex functions.
The quadratic regularization of two--stage linear stochastic optimization problems was developed by Ruszczy{\'n}ski \cite{ruszczynski1986regularized, ruszczynski1997accelerating}.
The same idea has also been implemented in the two--stage and multistage versions of the Stochastic Decomposition method developed by Higle and Sen  
\cite{ higle1994finite,sen2014multistage}, as well as the decomposition work of Morton \cite{morton1996enhanced}.
All of these methods utilize a scenario tree, either explicitly or implicitly (when indexing regularization terms by the entire history $H_t$), and consider separate incumbent solutions for every parent node in the scenario tree.
That limits their applicability to problems with short time horizons. On the other hand, the method described below is universal and can be applied to problems with a
large number of time periods.

\section{Problem Formulation}
\label{s:2}
Given a probability space $(\varOmega,\F,P)$ with a sigma--algebra $\F$, and
a filtration $\{\emptyset,\varOmega\}=\F_1 \subset \F_2 \subset\dots\subset\F_T = \F$, we consider a
stochastic process $\{W_t,~ t=1,\dots,T\}$ adapted to $\{\F_t,~t=1,\dots,T\}$.  Throughout our presentation, we adopt the convention that any variable indexed by $t$ is $\F_t$--measurable (surprisingly, this is not a standard assumption).
Our goal is to develop new solution methods for the following multistage linear stochastic programming problem:
\begin{equation}
\label{eq:1}
\min_{\mathrlap{\substack{A_0x_0 = b_0\\x_0\geq 0}}} \langle c_0, x_0\rangle + \E_1\left[\min_{\mathrlap{\substack{B_0x_0 + A_1x_1 = b_1\\x_1\geq 0}}} \langle c_1, x_1\rangle + \E_2\left[\dots +
\E_T\left[\min_{\mathrlap{\substack{B_{T-1}x_{T-1}+A_Tx_T=b_T\\ x_T\geq 0}}} \qquad\langle c_{T}, x_T\rangle\qquad\right]\dots\right] \right].
\end{equation}
The components of the information process $W_t = (A_t, B_t,b_t, c_t), t=1,\dots,T$ are the $\F_t$--measurable random matrices $A_t, B_t$ and vectors $b_t, c_t$,
while $A_0, B_0, b_0, c_0$ are assumed to be deterministic components of the initial state of the system $S_0 = (A_0, B_0, b_0, c_0)$.
We denote the sets of possible realizations of $W_t$ with $\Omega_t ,~ t=1,\dots, T$.
Those correspond to nested partitions of $\Omega$ given by the filtration $\{\F_t,~t=1,\dots,T\}$, and each $w\in\Omega$ can be represented as $\omega = (\omega_1, \omega_2,\dots,\omega_T)\in \Omega_1\times\Omega_2\times\dots\times\Omega_T$.
We assume that each sample set $\Omega_t$ has a finite number of elements that is small enough to be enumerated computationally.

\begin{definition}
The information history at time $t$ is $H_t = \{S_0, \omega_1,\omega_2,\dots, \omega_t \},$ where $H_t\in\Hc_t = \{S_0\}\times\Omega_1\times\Omega_2\times,\dots\times\Omega_t$.
Further, we define the post--decision information history at time $t$ to be $H_t^x = \{S_0, x_0, \omega_1, x_1,\omega_2, x_2, \dots, \omega_t, x_t\}$.
\end{definition}
Employing a dynamic programming framework, we distinguish between two types of states of the system, the pre--decision states $S_t$ and the post--decision states $S_t^x$.
\begin{definition} The (pre--decision) state $S_t$ of the system at time $t\geq 1$ is all the information in $ H^x_{t-1} \cup \omega_t$
that is necessary and sufficient to make a decision at time $t$,
and model the impact of $H_{t-1}^x \cup \omega_t$ on the computation of costs, constraints and transitions from time $t$ onward.
\end{definition}

Furthermore, the pre--decision state of the system $S_t$ can be represented as $S_t = (R_t, I_t)$, where
the pre--decision resource state $R_t$ is the amount of resources available at the beginning of time period $t$,
and $I_t$ is the pre--decision information state.
Please note that $R_t$ depends on both the decision $x_{t-1}$ and the random vector $b_t$,
$$ R_t = B_{t-1}x_{t-1} - b_t.$$
The information state $I_t$ contains all the remaining information of $S_t$ that is necessary and sufficient to model the system but is not in $R_t$.
Formally, we consider the following model for the evolution of the system over time:
$$S_0\xrightarrow{x_0} S^x_0\xrightarrow{\omega_1} S_1\xrightarrow{x_1} S^x_1\xrightarrow{\omega_2}  \ldots \xrightarrow{\omega_T} S_T \xrightarrow{x_T} S^x_T.$$

\begin{definition} The post--decision state $S_t^x, t\geq 0$ of the system at time $t$ is all the information in the post--decision history $H_t^x$
that is necessary and sufficient to model the impact of $H_t^x$ on the computation of costs, constraints and transitions from time $t$ onward, after a decision has been  made.
\end{definition}

We also represent the post--decision state of the system as $S_t^x = (R_t^x, I_t^x)$.
The \emph{post--decision resource state} $R_t^x$ is given by $$R_t^x = B_t x_t,$$ and the
\emph{post--decision information state} $I_t^x$ represents all the information in $S_t^x$
that is not in $R_t^x$.
Moreover, we refer to the rank of the matrix $B_t$ as the \emph{dimension of the post--decision resource state}.
If we define $$C(S_t, x_t) := \langle c_t, x_t\rangle$$ and the set $\Xc_t(S_t)$ is such that the following conditions are satisfied,
$$\Xc_t(S_t) :=\left\{
  \begin{array}{lr}
  x_t \in \R^{n_t}~:~ A_tx_t = b_t, &\mbox{ if } t = 0\\
  x_t \in \R^{n_t}~:~ B_{t-1}x_{t-1} + A_tx_t = b_t, &\mbox{ if }  t > 0
 \end{array}
  \right.
$$
then we can rewrite problem (\ref{eq:1}) using dynamic programming notation as follows,
\begin{equation}
\label{eq:2}
\min_{\mathrlap{\substack{x_0\in\Xc_0(S_0)}}} C(S_0, x_0)+ \E_1\left[\min_{\mathrlap{\substack{x_1\in\Xc_1(S_1)}}} C(S_1, x_1) + \E_2\left[\dots +
\E_T\left[\min_{\mathrlap{x_T\in\Xc_T(S_T)}} C(S_T, x_T)\right]\dots\right] \right].
\end{equation}
Since the problem is stochastic, its optimal solution is not a vector but rather a policy $\pi$,
which is a function that maps states $S_t$ to decisions $x_t\in \X_t(S_t)$.
Thus, we can consider the optimization problem (\ref{eq:2}) to be a search for an optimal policy $\pi^*$ over the set $\Pi$ consisting of all feasible and implementable policies
\begin{equation}
\label{eq:3}
\min_{\pi\in\Pi}\E\left[\sum_{t=0}^T C(S_t, X_t^\pi(S_t))\right].
\end{equation}
We refer to equation (\ref{eq:3}) as the \emph{base model}, and we can solve it by constructing an optimal lookahead policy.  
In that case, the optimal decisions $X^*_t(S_t)$ corresponding to $\pi^*$ satisfy the following optimality equation:
\begin{eqnarray}
X^*_t(S_t) &\in& \argmin_{x_t\in \Xc_t(S_t)} \left(C(S_t,x_t) + \min_{\pi\in\Pi}\E\left\{\left.\sum_{t'=t+1}^T C(S_{t'},X^\pi_{t'}(S_{t'}))\right|S_{t}^x\right\}\right).
\label{eq:lookaheadpi}
\end{eqnarray}
Therefore, we can also specify an optimal lookahead policy $\pi^*$ by employing its corresponding post--decision value functions $V^*_t(S_t^x)$, 
\begin{eqnarray}
V^*_t(S_{t}^x)= \min_{\pi\in\Pi}\E\left\{\left.\sum_{t'=t+1}^T C(S_{t'},X^\pi_{t'}(S_{t'}))\right|S_{t}^x\right\}
\end{eqnarray}
\begin{remark}
It is common for practitioners to employ a scenario tree in order to construct an approximate lookahead policy for problem (\ref{eq:2}) as follows: 
\begin{eqnarray}
X^*_t(S_t) &\in& \argmin_{x_t\in \Xc_t(S_t)} \left(C(S_t,x_t) + \min_{\pi\in\Pi}\E\left\{\left.\sum_{t'=t+1}^{t''} C(S_{t'},X^\pi_{t'}(S_{t'}))\right|S_{t}^x\right\}\right).
\label{eq:approxlookaheadpi}
\end{eqnarray}
When $t'' < T$ the optimality of the approximate lookahead policy given by (\ref{eq:approxlookaheadpi}) cannot be established. In addition, lower and upper bounds to the optimal value of problem (\ref{eq:2}) 
might not be readily available  (due to approximation errors stemming from the stage reduction).
Nonetheless, lookahead models can produce high-quality solutions in selected problems (see \cite{gulten2014two}).  
\end{remark}
Thus, at any time period $t = 0, \dots, T$, the optimal decision $X_t^*(S_t)$ for problem (\ref{eq:2}) can be computed as
$$ X_t^*(S_t)\in \argmin_{x_t\in\Xc_t(S_t)}\{ C(S_t, x_t) + V^*_t(S_{t}^x)\}.$$
Hence, the search for an optimal policy $\pi^*$ is equivalent to the computation of optimal post--decision value functions $V^*_t(S^x_t), t=0,\dots, T$.
One of the well--known methods that allows us to construct such value functions
is Stochastic Dual Dynamic Programming.

\section{Stochastic Dual Dynamic Programming}
\label{s:sddp}
\label{s:3}
Typically, stochastic programming techniques model the flow of information by utilizing
a scenario tree that involves the entire set $\Hc_t= \{S_{0}\}\times\Omega_{1}\times\dots\times \Omega_{t}$.
While such an approach is useful for analytical purposes, its practical applicability is limited by the
exponential growth of the number of nodes in the scenario tree when the length of the time horizon $T$ increases.
To overcome that challenge, Pereira and Pinto \cite{PereiraPinto} introduced the Stochastic Dual Dynamic Programming (SDDP) method for the solution of multistage stochastic linear optimization problems over long time horizons.
SDDP overcomes the combinatorial explosion of the information state by exploiting (a key and limiting assumption of) stagewise independence as
$\Pb(\omega_{t+1}|H_{t}) = \Pb(\omega_{t+1})$, and therefore all post--decision states $S^x_t$ share a single information state $I^x_t$.
Hence, $S_t^x = R_t^x$ and the optimal value functions $V^*_t(S_t^x)$ only depend on the post--decision resource states $R^x_t$,
$$ V^*_t(S^x_t) = V^*_t(R^x_t),~ t=0,\dots, T.$$
The convexity of the optimal value functions $V^*_t(R^x_t)$ is the key property that allows one to partially escape the curse of
dimensionality arising from partitioning the resource space.
Instead, $V^*_t(R^x_t)$ can be approximated with lower--bounding convex outer approximations $\overline{V}_t^k(R^x_t)$
whose functional form is the maximum over a collection of affine functions.
Those are commonly known as cutting hyperplanes or \emph{Benders cuts}, and are constructed at the resource points $R^{x,j}_t$ that are visited during the $j$--th forward pass,
\begin{equation}
\label{eq:vfa}
\overline{V}_{t}^{k}(R^x_t) :=  \max_{j \leq k}\big\{\alpha^j_{t} + \langle\beta^j_{t}, R^x_t - R^{x,j}_t\rangle\}.
\end{equation}
For example, at iteration $k$ we would obtain $R^{x,k}_t$ by solving the following linear program
\begin{equation}
\label{eq:4}
 x^k_t \in 
 \argmin_{x_t\in\Xc_t(S_t)}~\left\{
C(S_t, x_t) + \overline{V}^{k-1}_{t}(R^x_t)\right\}
 ,~ \mbox{ and setting } R_{t}^{x,k} \gets B_t^k x_t^k
\end{equation}
where $\overline{V}^0_t(R^x_t) = 0.$ 

The approximations $\overline{V}^{k-1}_t(R^x_t)$ are updated in the backward pass of iteration $k$
by constructing a cutting hyperplane $h_{t}^k(R^x_{t})$ to the optimal value function $V^*_t(R^x_t)$. To accomplish this,
we use a lower bound $\underbar{V}^k_{t+1}(R^{x,k}_t)$ (derived from solutions to subproblems for time $t+1$) to $V^*_t(R^{x,k}_t)$,
\begin{equation}
\label{eq:h}
\displaystyle h_{t}^k(R^x_{t}) :=  \underline{V}^k_{t+1}(R^{x,k}_{t})  + \langle\beta^k_{t}, R^x_{t} - R^{x,k}_{t}\rangle.
\end{equation}
Please note that the hyperplane $h^k_t(R_{t}^{x})$ is not necessarily tangent to $V^*_t(R^x_t)$ since\\
$\underbar{V}_{t+1}^k(R^{x,k}_t)$ might be strictly smaller than $V^*_t(R^{x,k}_t)$.
\begin{remark} When we need to emphasize the dependence of the feasible set $\Xc_{t+1}(S_{t+1})$ on the previous post--decision state $R^x_t$,
we use the notation $\Xc_{t+1}(R^x_t, I_{t+1})$, where the exogenous information in $R_{t+1}$ that is not contained in $R^x_t$ is assumed to be contained in $I_{t+1}$.
\end{remark}
In order to construct $\underbar{V}_{t+1}^k $,
we consider every element of the sample set $\omega_{t+1}\in\Omega_{t+1}$ and denote with $\underline{V}^k_{t+1}(R^x_{t}, \omega_{t+1})$
the optimal value of the following optimization problem,

\begin{equation*}
  \underline{V}^k_{t+1}(R_{t}^{x}, \omega_{t+1}) := 
  \min_{x_{t+1}\in\Xc_{t+1}(R_{t}^{x}, I_{t+1}(\omega_{t+1}))}~\left\{C(S_{t+1}(\omega_{t+1}), x_{t+1}) + \overline{V}^{k}_{t+1}(R^x_{t+1})
  \right\}.
\end{equation*}
Finally, we set
\begin{equation*}
  \underline{V}^k_{t+1}(R_{t}^{x}) := \sum_{\omega_{t+1}\in\Omega_{t+1}} \Pb(\omega_{t+1})\underline{V}^k_{t+1}(R^x_{t}, \omega_{t+1}).
\end{equation*}
Hence, if we choose
\begin{equation*}
 \beta^k_{t}\in \partial_R \underline{V}^k_{t+1}(R_{t}^{x,k}),
\end{equation*}
%
then we can construct a new aggregated cut $\displaystyle h_{t}^k(R^x_{t})$ as described in equation (\ref{eq:h}).
Thus, in the backward pass of iteration $k$, we can update the approximate value function $\overline{V}_{t}^{k}(R^x_{t})$ as follows,
  $$\displaystyle \overline{V}_{t}^{k}(R^x_{t}) := \max\big\{\overline{V}_{t}^{k-1}(R^x_{t}),~h_{t}^k(R^x_{t})\big\}.$$

%

If none of the constructed cuts are removed, then the growing collections of affine functions
generate sequences of \emph{monotonically increasing lower bounding approximations} $\overline{V}_t^k(R^x_t)$ to the optimal post--decision value functions
$V^*_t(R^x_t)$ for any $t=0,\dots, T-1$.
$$\overline{V}^{k-1}_t(R^x_t) \leq  \overline{V}^{k}_t(R^x_t) \leq V^*_t(R^x_t), \forall k\in \N, t=0,\dots, T-1.$$
Furthermore, in this work we assume relatively complete recourse, i.e. for any feasible solutions to the optimization problems at time periods $t=0,\dots, T-1$, there exists a
feasible solution to any realized stage $t+1$ subproblem with probability one. This assumption alleviates the need for feasibility cuts and allows us to
improve the clarity of the presentation.

\section{Quadratic Regularization}
\label{s:4}
Existing regularization approaches \cite{ruszczynski1997accelerating,ruszczynski1986regularized, higle1994finite,sen2014multistage, morton1996enhanced} utilize a scenario tree and consider separate incumbent solutions
$\bar{x}_t(H_t)$ for every possible history ${H}_{t}\in\Hc_t, t=0,\dots,T-1$.
The underlying idea in such methods has been to augment optimization problems of the form (\ref{eq:4}) with a regularization term as follows,
 \begin{equation}
\label{eq:5}
 x^k_t \in 
 \argmin_{x_t\in\Xc_t(S_t)}~ \left\{C(S_t, x_t) + \overline{V}^{k-1}_{t}(R^x_t) + \frac{\rho}{2}||{x_t} - \bar{x}_t(H_t)||_2^2\\
 \right\}.
\end{equation}

As the algorithm progresses, each incumbent solution $\bar{x}_t(H_t)$ is updated to a new optimal solution, if certain conditions are satisfied.
While such an approach is feasible for problems on scenario trees with small $T$,
it is not practical for non--trivial time horizons. The exponential growth of the scenario tree ensures that only a tiny fraction of all possible realizations
$H_t\in\Hc_t,~t=0,\dots, T-1$ could be examined in the forward pass in a reasonable computational time. Moreover, multiple visits to each $H_{t}\in\Hc_t,~t=0,\dots, T-1$
and multiple updates of its incumbent solution are also out of the realm of computational feasibility for most practical instances.
One way to remedy this difficulty would be for different histories to share incumbent solutions.
For example, a single incumbent solution $\bar{x}_t$ can be shared among all realizations $H_t\in\Hc_t$ and that would result in the optimization problem
 \begin{equation}
\label{eq:6}
 x^k_t \in  
 \argmin_{x_t\in\Xc_t(S_t)}~ \left\{
C(S_t, x_t) + \overline{V}^{k-1}_{t}(R^x_t) + \frac{\rho}{2}||x_t - \bar{x}_t||_2^2\\
 \right\}.
\end{equation}

Equation (\ref{eq:6}) can be used in place of equation (\ref{eq:4}), and it would still result in a convergent method for a fixed set of incumbent solutions
$\bar{x}_t, t=0,\dots,T-1$.
However, the optimality of the resulting policy cannot be established.
Moreover, since the purpose of the quadratic regularization term is to mitigate the inaccuracy of the value function approximations, we do not need to regularize
around the entire vector $x_t$ (which might be very high--dimensional) but only around the parameters $R^x_t$ of the post--decision value function approximations $\overline{V}^{k-1}_t(R^x_t)$.
Thus, we can adjust problem (\ref{eq:6}) to address these concerns by making the following adjustments,
 \begin{equation}
\label{eq:7}
 x^k_t \in
 \argmin_{x_t\in\Xc_t(S_t)}~  \left\{C(S_t, x_t) + \overline{V}^{k-1}_{t}(R^x_t) + \frac{\varrho^k}{2}\Big\langle R^x_t - \overline{R}^{x,k-1}_t, Q_t(R^x_t - \overline{R}^{x,k-1}_t)\Big \rangle\\
 \right\}
\end{equation}
where the sequence of penalty coefficients $\{\varrho^k\}$ is such that
$\varrho^k \geq 0, ~\forall k\in\N$ and $\displaystyle \lim_{k\rightarrow\infty}\varrho^k = 0$. We also introduce a positive semi--definite matrix $Q_t \succeq 0$, which
can be used to address any scaling concerns across different entries of the resource vectors $R^x_t$.
Please note that the meaning of the proposed regularization strategy is quite different from its scenario tree counterparts, as it aims to steer the solution towards a ``known'' region of the value function domain, rather than to the ``correct'' solution for the given history $H_t$ of the stochastic process. Hence, we choose the incumbent solutions to be the previous points encountered in the forward pass since the cuts supported at those points are the ones generated with the most information.
Finally, we also point out that unlike the case of scenario trees, in the current method we do not aim for the convergence of the incumbent solutions towards any point. Interested readers are free to choose different incumbent solutions that they might find appropriate, and the convergence results presented below would still hold.

Now, we can substitute equation (\ref{eq:7}) for equation (\ref{eq:4}) in the forward pass of SDDP, and the new procedure would still converge to an optimal
solution of problem (\ref{eq:2}) with probability one after a finite number of iterations. That might appear somewhat surprising since gradient methods applied to quadratic optimization problems typically entail asymptotic convergence. However, in this case a finite number of iterations is sufficient since the true problem remains linear, and the quadratic terms are only used to guide the exploration phase of the forward pass. Moreover, the generation of the supporting hyperplanes in the backward pass utilizes linear programming problems which can generate only a finite number of different cuts when basic dual feasible solutions are used. 
The details of the resulting method are presented in Algorithm \ref{a:1}, and we study its convergence properties below.

\begin{algorithm}
	\caption{Quadratic Regularization Method with Stagewise Independence}
\footnotesize
\begin{algorithmic}[1]
\STATE Choose $Q_t\succeq 0, t=0,\dots,T$, and define sequence $\{\varrho^k\}$.
\STATE Define $\overline{V}_{T}^k(R^x_T):= V^*_T(R^x_T), k=0,\dots, K$.
\STATE Define $\overline{V}_{t}^0(R^x_t):=-\infty, t=0,\dots, T-1$.
\STATE $(R^{x,k}_{-1}, I_0) \gets S_0,~ k = 0,\dots, K$
\FOR{$k=0,\dots, K$}
\STATE \emph{Forward Pass:}
\STATE Sample $\omega\in\Omega$.
\FOR{$t=0,\dots,T$}
\IF{$(k = 0)$}
\STATE
\begin{equation*}
 \mbox{Select }x^k_t \in 
 \argmin_{x_t\in\Xc_t(R^{x,k}_{t-1}, I_t(\omega))}~ \left\{C(S_t(\omega), x_t)
 \right\}
\end{equation*}

\ELSE
\IF{$t < T$}
\STATE

\begin{equation*}
 x^k_t \in 
 \argmin_{x_t\in\Xc_t(R^{x,k}_{t-1}, I_t(\omega))}~ \left\{C(S_t(\omega), x_t) + \overline{V}^{k-1}_{t}(R^x_t) +
      \frac{\varrho^k}{2} \Big\langle R^x_t - \overline{R}^{x,k-1}_t, Q_t(R^x_t - \overline{R}^{x,k-1}_t)\Big\rangle
 \right\}
 \end{equation*}


\ELSE
\STATE

\begin{equation*}
\mbox{Select } x^k_t \in 
 \argmin_{x_t\in\Xc_t(R^{x,k}_{t-1}, I_t(\omega))}~ \left\{C(S_t(\omega), x_t) + \overline{V}^{k-1}_{t}(R^x_t) 
 \right\}
 \end{equation*}


\ENDIF
\ENDIF
\STATE $\mbox{Set } R_{t}^{x,k} \gets B_t^k x_t^k $;  $S_{t+1}(\omega) \gets (R_t^{x,k} - b_{t+1}(\omega), I_{t+1}(\omega))$
\ENDFOR
\STATE \emph{Backward Pass:}
\FOR{$t=T,\dots, 1$}
\STATE \begin{equation*}
  \mbox{Define } \underline{V}^k_t(R^{x}_{t-1}, \omega_t) := 
  \min_{x_t\in\Xc_t(R^{x}_{t-1}, I_t(\omega_t))}~ \left\{C(S_t(\omega_t), x_t) + \overline{V}^k_{t}(R^x_t)
  \right\} 
  \end{equation*}
  \FORALL{$\omega_t\in\Omega_t$}
\STATE \begin{equation*}
  \mbox{Select }\underline{\beta}^k_t(\omega_t) \in \partial_{R^x_{t-1}} \underline{V}^k_t(R_{t-1}^{x,k}, \omega_t) 
  \end{equation*}
  \ENDFOR
  \STATE $\displaystyle \alpha^k_{t-1}\gets \sum_{\omega_t\in\Omega_t} \Pb(\omega_t)\underline{V}^k_t(R^{x,k}_t, \omega_t)$; $\displaystyle \beta^k_{t-1} \gets \sum_{\omega_t\in\Omega_t} \Pb(\omega_t) \underline{\beta}_{t}^k(\omega_t)$
  \STATE $\displaystyle h_{t-1}^k(R^{x}_{t-1}) :=  \alpha^k_{t-1} + \langle\beta^k_{t-1}, R^x_{t-1} - R^{x,k}_{t-1}\rangle$
  \STATE $\displaystyle \overline{V}_{t-1}^{k}(R^x_{t-1}) := \max\big\{\overline{V}_{t-1}^{k-1}(R^x_{t-1}),~h_{t-1}^k(R^x_{t-1})\big\}$
\ENDFOR
\STATE
  \begin{fleqn} \begin{equation*}
  \underline{V}^k_0 \gets \left\{
  \begin{split}
  \min_{x_0\in\Xc_0(S_0)}~&C(S_0, x_0) + \overline{V}^k_{0}(R^x_0)\\
  \end{split}
  \right\} 
  \end{equation*}
  \end{fleqn}
\STATE $\overline{R}^{x,k}_t \gets R^{x,k}_{t},~t=0,\dots, T-1$
\ENDFOR
\end{algorithmic}
\label{a:reg-sddp}
\label{a:1}
\end{algorithm}
\normalsize

\begin{lemma}[\cite{philpott2008convergence, shapiro2011analysis}] Suppose that dual basic solutions are used in the solution of subproblems in the backward pass of Algorithm \ref{a:1}.
Then, there exist a finite number of possible value function approximations $\overline{V}_t(\cdot), t=0,\dots,T$.
\label{l:3}
\end{lemma}


Since the regularization terms are artificial for the original problem, we exclude them from the definition of an optimal policy.
\begin{definition}
The value function approximations $\overline{V}_t^k, t = 0,\dots,T$ are optimal for problem (\ref{eq:2}) if for any realization $\omega\in\Omega$,
\begin{equation}
\label{eq:8}
 \min_{x_t\in\Xc_t(S_t(\omega))}~ \big\{C(S_t(\omega), x_t) + \overline{V}^{k}_{t}(R^x_t)\big\} =
\min_{x_t\in\Xc_t(S_t(\omega))}~ \big\{C(S_t(\omega), x_t) + V^*_{t}(R^x_t)\big\}
\end{equation}
for $t=0,\dots,T$.
\end{definition}
\begin{theorem} Suppose that Algorithm \ref{a:reg-sddp} satisfies the following assumptions:
\label{t:7}
\begin{enumerate}
\item $\overline{V}_T^k(\cdot) \equiv V^*_T(\cdot),~k\in\N.$ \label{as:1}
\item Dual basic optimal solutions are used in the backward pass. \label{as:2}
\item Every element $\omega\in \Omega$ has a strictly positive probability $\Pb(\omega) > 0$.\label{as:3}
\item $\varrho^k \geq 0$ and $\displaystyle \lim_{k\rightarrow \infty} \varrho^k = 0$.
\item The feasible sets $\Xc_t(S_t)$ are bounded for each $t=0,\dots,T$.\label{as:5}
\end{enumerate}
Then, the regularization method presented in Algorithm \ref{a:1} converges to an optimal policy of problem (\ref{eq:2}) after a finite number of iterations
with probability one.
\end{theorem}
\begin{proof}{Proof:}
Let $\mathcal{\overline{V}}_t$ denote the set of all possible value function approximations\\
$\overline{V}^k_t,~t= 0,\dots,T$ that can be generated by the backward pass of Algorithm \ref{a:1}. 
Since according to Assumption \ref{as:2} we use only dual basic optimal solutions in the backward pass, by Lemma \ref{l:3} we know that the sets $\mathcal{\overline{V}}_t$ have finite cardinality for all $t=0,\dots,T$.
Thus, we know that as the algorithm progresses all the value function approximations $\overline{V}^k_t$ will eventually stabilize. Therefore, there exists an iteration index $k_1\in\N$ after which no updates can be made to the value functions $\overline{V}^k_t, ~t = 0,\dots,T$ for $k>k_1$. If the value functions $\overline{V}^{k_1}_t, ~t = 0,\dots,T$ are optimal for problem (\ref{eq:2}), then we are done.

Now, suppose that was not the case. Then there exists $t^\prime, ~0\leq t^\prime < T$, and $\omega^\prime\in\Omega$ such that for any $k>k_1$ we have 
$$ \min_{\mathclap{x_{t^\prime}\in\Xc_{t^\prime}(S_{t^\prime}(\omega^\prime))}}~ \big\{C(S_{t^\prime}(\omega^\prime), x_{t^\prime}) + {V}^{*}_{t^\prime}(R^x_{t^\prime}(\omega^\prime))\big\} >
\min_{\mathclap{x_{t^\prime}\in\Xc_{t^\prime}(S_{t^\prime}(\omega^\prime))}}~\big\{C(S_{t^\prime}(\omega^\prime), x_{t^\prime}) + \overline{V}^{k-1}_{t^\prime}(R^x_{t^\prime}(\omega^\prime))\big\}$$

Let us consider the set
\begin{equation}
\begin{split}
&\Delta = \Big\{\delta\in\R,\\
&\delta = \min_{\mathclap{x_t\in\Xc_t(S_t(\omega))}}~ \big\{C(S_t(\omega), x_t) + {V}^{*}_{t}(R^x_t(\omega))\big\} - 
\min_{\mathclap{x_t\in\Xc_t(S_t(\omega))}}~ \big\{C(S_t(\omega), x_t) + \overline{V}_{t}(R^x_t(\omega))\big\}: \\
&\min_{\mathclap{x_t\in\Xc_t(S_t(\omega))}}~ \big\{C(S_t(\omega), x_t) + {V}^{*}_{t}(R^x_t(\omega))\big\} >
\min_{\mathclap{x_t\in\Xc_t(S_t(\omega))}}~ \big\{C(S_t(\omega), x_t) + \overline{V}_{t}(R^x_t(\omega))\big\},\\
&\mbox{where } \omega\in\Omega,\mbox{ and } \overline{V}_t \in \mathcal{\overline{V}}_t, t=0,\dots, T\Big\}
\end{split} 
\end{equation}
Since the number of elements $\omega\in\Omega$ is finite, we know that the set $\Delta$ also has a finite number of elements. Thus, $\Delta$ has a minimum element, and we denote
$$\epsilon = \min\Delta.$$ 
Hence,
$$ \min_{\mathclap{x_{t^\prime}\in\Xc_{t^\prime}(S_{t^\prime}(\omega^\prime))}}~ \big\{C(S_{t^\prime}(\omega^\prime), x_{t^\prime}) + {V}^{*}_{t^\prime}(R^x_{t^\prime}(\omega^\prime))\big\} -
\min_{\mathclap{x_{t^\prime}\in\Xc_{t^\prime}(S_{t^\prime}(\omega^\prime))}}~ \big\{C(S_{t^\prime}(\omega^\prime), x_{t^\prime}) + \overline{V}^{k-1}_{t^\prime}(R^x_{t^\prime}(\omega^\prime))\big\} \geq \epsilon$$
And if $t^{\prime} > 0$ we know that
$$V^*_{t^\prime - 1}(R^x_{t^\prime - 1}) = \sum_{\omega_{t^\prime}\in\Omega_{t^\prime}}\Pb(\omega_{t^\prime})\min_{x_{t^\prime} \in \Xc(R^x_{t^\prime - 1}, I_t(\omega_{t^\prime}))}\big\{ C(S_{t^\prime}(\omega_{t^\prime}), x_{t^\prime}) + {V}^{*}_{t^\prime}(R^x_{t^\prime})\big\} $$
and using convexity,
$$\overline{V}^{k - 1}_{t^\prime - 1}(R^x_{t^\prime - 1}) \leq \sum_{\omega_{t^\prime}\in\Omega_{t^\prime}}\Pb(\omega_{t^\prime})\min_{x_t^\prime \in \Xc(R^x_{t^\prime - 1}, I_t(\omega_{t^\prime}))}\big\{ C(S_{t^\prime}(\omega_{t^\prime}), x_{t^\prime}) + \overline{V}^{k - 1}_{t^\prime}(R^x_{t^\prime})\big\}.$$
Therefore,
\begin{equation}
\begin{split}
&\min_{x_{t^\prime - 1}\in\Xc_{t^\prime - 1}(S_{t^\prime - 1}(\omega^\prime))}~ \big\{C(S_{t^\prime - 1}(\omega^\prime), x_{t^\prime - 1}) + {V}^{*}_{t^\prime - 1}(R^x_{t^\prime - 1}(\omega^\prime))\big\}\\
&>\min_{x_{t^\prime - 1}\in\Xc_{t^\prime - 1}(S_{t^\prime - 1}(\omega^\prime))}~ \big\{C(S_{t^\prime - 1}(\omega^\prime), x_{t^\prime - 1}) + \overline{V}^{k - 1}_{t^\prime - 1}(R^x_{t^\prime - 1}(\omega^\prime))\big\},
\end{split}
\end{equation}
which implies
\begin{equation}
\begin{split}
&\min_{x_{t^\prime - 1}\in\Xc_{t^\prime - 1}(S_{t^\prime - 1}(\omega^\prime))}~ \big\{C(S_{t^\prime - 1}(\omega^\prime), x_{t^\prime - 1}) + {V}^{*}_{t^\prime - 1}(R^x_{t^\prime - 1}(\omega^\prime))\big\}\\
&-\min_{x_{t^\prime - 1}\in\Xc_{t^\prime - 1}(S_{t^\prime - 1}(\omega^\prime))}~ \big\{C(S_{t^\prime - 1}(\omega^\prime), x_{t^\prime - 1}) + \overline{V}^{k - 1}_{t^\prime - 1}(R^x_{t^\prime - 1}(\omega^\prime))\big\}\geq \epsilon.
\end{split}
\end{equation}
Proceeding by backward induction, we know that
$$ \min_{x_{0}\in\Xc_{0}(S_{0})}~ \big\{C(S_{0}, x_{0}) + {V}^{*}_{0}(R^x_{0})\big\} -
\min_{x_{0}\in\Xc_{0}(S_{0})}~ \big\{C(S_{0}, x_{0}) + \overline{V}^{k - 1}_{0}(R^x_{0})\big\} \geq \epsilon. $$
Moreover, using Assumption \ref{as:5} we know that $R^x_t$ is bounded for each $t$. Hence, without loss of generality we can assume that $k$ is such that
$$\varrho^{k}\langle R^x_t - \overline{R}^{x,k-1}_t,Q_t (R^x_t - \overline{R}^{x, k-1}_t)\rangle < \epsilon,~ \mbox{ for } t = 0,\dots,T-1.$$
Hence, if we denote with $\widetilde{x}_0^{k}$ the solution to the following regularized problem, 
$$\widetilde{x}_0^{k}=  \argmin_{\mathclap{x_{0}\in\Xc_{0}(S_{0})}}~ \big\{C(S_{0}, x_{0}) + \overline{V}^{k - 1}_{0}(R^x_{0}(\omega^\prime)) + \varrho^{k}\langle R^x_0(\omega^\prime) - \overline{R}^{x,k - 1}_0,Q_0 (R^x_0(\omega^\prime) - \overline{R}^{x, k - 1}_0)\rangle \big\}$$
then we know that
$$\min_{\mathclap{x_{0}\in\Xc_{0}(S_{0})}}~ \big\{C(S_{0}, x_{0}) + {V}^{*}_{0}(R^x_{0})\big\} > C(S_{0}, \widetilde{x}_{0}^{k}) + \overline{V}^{k - 1}_{0}(R^{\widetilde{x},{k}}_0)+ \varrho^{k}\langle R^{\mathrlap{\widetilde{x},k}}_0 - \overline{R}^{{x,k-1}}_0,Q_0 (R^{\widetilde{x},k}_0 - \overline{R}^{\mathrlap{x, k - 1}}_0)\rangle$$
And since $Q$ is positive semi--definite, we know that
$$ \min_{x_{0}\in\Xc_{0}(S_{0})}~ \big\{C(S_{0}, x_{0}) + {V}^{*}_{0}(R^x_{0})\big\} > C(S_{0}, \widetilde{x}_{0}^{k}) + \overline{V}^{k-1}_{0}(R^{\widetilde{x},{k}}_0) $$
which implies,
$$ C(S_{0}, \widetilde{x}^{k}_0) + {V}^{*}_{0}(R^{\widetilde{x}, k}_{0}) > C(S_{0}, \widetilde{x}_{0}^{k}) + \overline{V}^{k - 1}_{0}(R^{\widetilde{x},{k}}_0).$$
and therefore
$${V}^{*}_{0}(R^{\widetilde{x}, k}_{0}) > \overline{V}^{k - 1}_{0}(R^{\widetilde{x},{k}}_0).$$
Thus we know that the value function approximation $\overline{V}^{k-1}_0(\cdot)$ is suboptimal at the point $R^{\widetilde{x}, k}_0$ corresponding to $\widetilde{x}^{k}_0$.
Hence, if the value function $V^{k-1}_1(\cdot)$ is such that for each $\omega_1\in\Omega_1$ the following holds, 
$$\min_{\mathrlap{x_{1} \in \Xc(R^{\widetilde{x}, k}_{0}, I_1(\omega_{1}))}}\big\{ C(S_{1}(\omega_{1}), x_{1}) + \overline{V}^{k-1}_{1}(R^x_{1})\big\} = \min_{\mathrlap{x_{1} \in \Xc(R^{\widetilde{x}, k}_{0}, I_1(\omega_{1}))}}\big\{ C(S_{1}(\omega_{1}), x_{1}) + {V}^{*}_{1}(R^x_{1})\big\}  $$
then the backward pass will result in an updated value function $\overline{V}^k_0(\cdot)$ such that $\overline{V}^{k}_0(R^{\widetilde{x}, k}_{0}) = V^*_0(R^{\widetilde{x}, k}_{0}) > \overline{V}^{k-1}_0(R^{\widetilde{x}, k}_{0}) $ which is a contradiction with the choice of $k$. Therefore, it must be the case that there exists $\omega_1^\dprime\in\Omega_1$ such that 
$$\min_{\mathrlap{x_{1} \in \Xc(R^{\widetilde{x}, k}_{0}, I_1(\omega_{1}^\dprime))}}\big\{ C(S_{1}(\omega_{1}^\dprime), x_{1}) + \overline{V}^{k-1}_{1}(R^x_{1})\big\} < \min_{\mathrlap{x_{1} \in \Xc(R^{\widetilde{x}, k}_{0}, I_1(\omega_{1}^\dprime))}}\big\{ C(S_{1}(\omega_{1}^\dprime), x_{1}) + {V}^{*}_{1}(R^x_{1})\big\}.$$
Moreover, 
$$\min_{x_{T} \in \Xc(S_T(\omega))}\big\{ C(S_{T}(\omega), x_{T}) + \overline{V}^{k-1}_{T}(R^x_{T})\big\} = \min_{x_{T} \in \Xc(S_T(\omega))}\big\{ C(S_{T}(\omega), x_{T}) + {V}^{*}_{T}(R^x_{T})\big\}.$$
Therefore, there exists a sample path $\omega^\dprime\in\Omega$ and a time index $t^\dprime,~0< t^\dprime < T$ such that the sequence of regularized solutions $\widetilde{x}_{t}^k(\omega^\dprime)$ would result in a suboptimal value function evaluation at $t^\dprime$, 
$$\min_{\mathrlap{x_{t^\dprime} \in \Xc(R^{\widetilde{x}, k}_{t^\dprime - 1}, I_{t^\dprime}(\omega^\dprime))}}\big\{ C(S_{t^\dprime}(\omega^\dprime), x_{t^\dprime}) + \overline{V}^{k-1}_{t^\dprime}(R^x_{t^\dprime})\big\} < \min_{\mathrlap{x_{t^\dprime} \in \Xc(R^{\widetilde{x}, k}_{t^\dprime - 1}, I_{t^\dprime}(\omega^\dprime))}}\big\{ C(S_{t^\dprime}(\omega), x_{t^\dprime}) + {V}^{*}_{t^\dprime}(R^x_{t^\dprime})\big\},$$
and optimal evaluations at $t^\dprime + 1$ for all possible $\omega_{t^\dprime+1}\in\Omega_{t^\dprime + 1}$,
\begin{equation*}
\begin{split}
\min_{x_{t^\dprime + 1} \in \Xc(R^{\widetilde{x}, k}_{t^\dprime}, I_{t^\dprime + 1}(\omega_{t_\dprime + 1}))}\big\{ C(S_{t^\dprime + 1}(\omega_{t^\dprime+1}), x_{t^\dprime + 1}) + \overline{V}^{k-1}_{t^\dprime + 1}(R^x_{t^\dprime + 1})\big\}\\
 = \min_{x_{t^\dprime + 1} \in \Xc(R^{\widetilde{x}, k}_{t^{\dprime}}, I_{t^\dprime + 1}(\omega_{t^\dprime+1}))}\big\{ C(S_{t^\dprime + 1}(\omega_{t^\dprime + 1}), x_{t^\dprime + 1}) + {V}^{*}_{t^\dprime + 1}(R^x_{t^\dprime + 1})\big\}.
\end{split}
\end{equation*}
Hence the backward pass of iteration $k$ will result in an updated value function approximation $\overline{V}^k_{t^\dprime}(\cdot)$ such that
$$\overline{V}^{k}_{t^\dprime}(R^{\widetilde{x}, k}_{t^\dprime}) = V^*_{t^\dprime}(R^{\widetilde{x}, k}_{t^\dprime}) > \overline{V}^{k-1}_{t^\dprime}(R^{\widetilde{x}, k}_{t^\dprime}),$$
which is a contradiction with the choice of $k$. This completes the proof.

\begin{figure}[H]
\begin{center}
\begin{tikzpicture}[baseline]
\pgfplotsset{every axis legend/.append style={at={(0.5,-0.25)},anchor=north, legend columns=2}}
\tikzset{
every pin/.style={fill=black!20!white,rectangle,rounded corners=3pt,font=\normalsize},
small dot/.style={fill=black,circle,scale=0.4}
}
\begin{axis}[
height=2.3in,
width=2.3in,
title={Forward Pass},
grid=major,
xtick = {},
ytick = {},
xmin=-1,
xmax=2,
ymin=0,
ymax=5.2,
xlabel={$R^x_{t^\prime}$},
ylabel={Value (\$)}]
\addplot[color=black, dashed, line width = 2] coordinates {
 (-1, 5) (-0.5, 2) (1, 2) (1.5, 5)
}
coordinate [pos=0.56] (C)
coordinate [pos=0.6] (B)
;
\addlegendentry{$\overline{V}^{k-1}_{t^\dprime}(R^x_{t^\dprime})$}
\addplot[color=blue, line width = 1] coordinates {
 (-1, 5) (-0.5, 2) (0.5, 2) (1.167, 3.167) (1.5, 5)
}
coordinate [pos=0.7] (A)
;
\addlegendentry{$V^*_{t^\dprime}(R^x_{t^\dprime}) $}
\node[small dot, pin=-120:{$R^{\tilde{x},k}_{t^\dprime}(\omega^\dprime)$}] at (170,200) {};
\draw[-stealth] (A) -| (B);
\draw[-stealth] (C) -| (B);
\end{axis}
\end{tikzpicture}
\begin{tikzpicture}[baseline]
\pgfplotsset{every axis legend/.append style={at={(0.5,-0.25)},anchor=north, legend columns=2}}
\tikzset{
every pin/.style={fill=black!20!white,rectangle,rounded corners=3pt,font=\normalsize},
small dot/.style={fill=black,circle,scale=0.4}
}
\begin{axis}[
height=2.3in,
width=2.3in,
title={Backward Pass},
grid=major,
xtick = {},
ytick = {},
xmin=-1,
xmax=2,
ymin=0,
ymax=5.2,
xlabel={$R^x_{t^\prime}$},
ylabel={Value (\$)}]
\addplot[color=black, dashed, line width = 2] coordinates {
 (-1, 5) (-0.5, 2) (0.5, 2) (1.167, 3.167) (1.5, 5)
}
coordinate [pos=0.56] (C)
coordinate [pos=0.6] (B)
;
\addlegendentry{$\overline{V}^{k}_{t^\dprime}(R^x_{t^\dprime})$}
\addplot[color=blue, line width = 1] coordinates {
 (-1, 5) (-0.5, 2) (0.5, 2) (1.167, 3.167) (1.5, 5)
}
coordinate [pos=0.7] (A)
;
\addlegendentry{$V^*_{t^\dprime}(R^x_{t^\dprime}) $}
\end{axis}

\end{tikzpicture}
\caption{Value function update at iteration $k$.}
\label{fig:proof}
\end{center}
\end{figure}
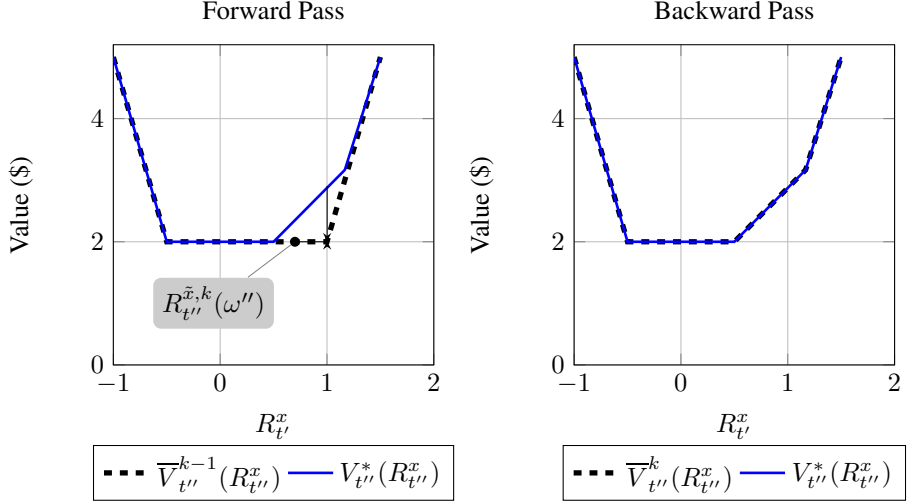
\end{proof}
\section{Beyond Stagewise Independence}
\label{s:5}
Despite its advantages, the SDDP methodology has one crucial drawback.
The stagewise independence of $W_t = (A_t, B_t, b_t, c_t)$ will generally not hold in practice since real--world multistage problems often involve
stochastic processes that exhibit some degree of temporal dependence. There are different approaches that we can adopt to address this difficulty.
First, let us consider the special case when the history dependence occurs only in the right hand side constraint vectors $b_t$, and it has the following autoregressive structure:
\begin{equation}
\label{eq:ar}
b_t = \sum_{t^\prime = 1}^{t-1} \left(\Phi_{t,t^\prime} b_{t^\prime} + \Psi_{t,t^\prime}\eta_{t^\prime}\right) + \eta_t
\end{equation}
where the process $(A_t, B_t, c_t, \eta_t)$ is stagewise independent and the deterministic matrices $\Phi_{t, t^\prime}$ and $\Psi_{t,t^\prime}$ contain the autoregressive information.
Then, for each time period $t > 0$ in the SDDP formulation, we can extend the original optimization problem
with additional variables to accommodate the realizations of $b_{t^\prime}$ and $\eta^{t^\prime}, t^\prime < t$ that are necessary to model the autoregressive dependence (see \cite{chiralaksanakul2004assessing, maceira2006use}, and \cite{shapiro2013risk}).
The advantage of such a solution to the history dependence problem is that stagewise independence is present in the extended formulation.
A drawback of the approach is that the dimension of the state space also increases from $|R^x_t|$ (in the stagewise independence case) to possibly as much as
$|R^x_t| + \sum_{t^\prime = 0}^{t-1}(|b_{t^\prime}| + |\eta_{t^\prime}|)$ (in the history dependent case), which implies a slower convergence rate (note that we can omit terms $|b_{t^\prime}|$ if $\Phi_{t, t^\prime} = \mathbf{0}$, and $|\eta_{t^\prime}|$ if $\Psi_{t,t^\prime}=\mathbf{0}$). This problem can be alleviated with the use of cut sharing strategies as described in \cite{infanger1996cut}, and \cite{de2013sharing}.

%
In the remainder of this section we consider an alternative setup that leads to an increase in the information dimension rather than the resource dimension.
We assume that the stochastic process $W_t$ is a discrete state Markov chain.
Thus, the probability of occurrence of $\omega_{t+1}\in\Omega_{t+1}$ depends only on the current realization $\omega_t\in\Omega_t$ or the current post--decision information state $I_t^x$,
\begin{equation}
\label{eq:14}
\Pb(\omega_{t+1}|H_{t}) = \left\{
  \begin{array}{lr}
  \Pb(\omega_{t+1}|S_t) = \Pb(\omega_{t+1}|I_t^x), &\mbox{ if } t = 0\\
  \Pb(\omega_{t+1}|\omega_t) = \Pb(\omega_{t+1}|I_t^x), &\mbox{ if } t > 0.
  \end{array}\right.
\end{equation}
Such an approach can be suitable for problems where the process $(A_t,B_t,c_t)$ is not stagewise independent, or the autoregressive model (\ref{eq:ar}) does not constitute a good fit
to the observed realizations of the random process.
For example, historical weather data might indicate the presence of distinct patterns that cannot be explained with a normal error distribution around a given mean (which arise in autoregressive estimation). Alternatively, the relevant information state could be the forecast of the highest temperature tomorrow.

To properly model such weather dynamics one might need to consider different weather regimes that are inherently distinct.
Thus, multiple approximations of the value functions need to be employed, which increases the size of the optimization problem.
That leads to greater computational requirements for solving the problem as a distinct recourse function approximation needs
to be constructed for every $I_t^x\in\Ic_t^x$, where $\Ic_t^x$ denotes the set of all possible post--decision information states at time $t$. 
Hence, we need to maintain $|\Ic_t^x(\Omega_t)|$ sets of cuts for each time period $t=0,\dots, T$, and therefore the approach is suitable for problems
where the cardinality of the possible post--decision information states $|\Ic_t^x(\Omega_t)|$ is small,
or alternatively when the cardinality of the sample sets $|\Omega_t|$ is small.
However, unlike the case of an autoregressive fit (\ref{eq:ar}), the dimension of the post--decision resource state is preserved in each set of cuts,
and the corresponding exponential increase in the computational time is avoided.

In the forward pass at iteration $k$, we consider a sample path $\omega = (\omega_1,\dots,\omega_T)$ that is generated using (\ref{eq:14}).
At each time step $t=0,\dots, T-1$ the piecewise--linear value function $V^{k-1}_t(R^x_t, I_t^x(\omega))$
is used to approximate the optimal value function $V^*_t(R^x_t, I^x_t(\omega))$ at the current post--decision information state $I^x_t(\omega)$.

In the backward pass of the algorithm at iteration $k$,
we consider $t=T,\dots, 1$ and generate the cutting hyperplanes $ h_{t-1}^k(R^x_{t-1}, I_{t-1}^x)$ for each $I_{t-1}^x\in\Ic_{t-1}^x$.
Please note that if the random process $W_t$ is a finite state Markov chain,
then $|\Ic_t^x(\Omega_t)| \leq |\Omega_t|, t=0,\dots, T$.
We employ the conditional probabilities $\Pb(\omega_t|I_{t-1}^x)$ to construct constant intercepts and slopes,
$$\displaystyle \alpha^k_{t-1}(I_{t-1}^x)\gets \sum_{\omega_t\in\Omega_t} \Pb(\omega_t|I_{t-1}^x)\underline{V}^k_t(R^{x,k}_t, \omega_t)$$
and,
$$\displaystyle \beta^k_{t-1}(I_{t-1}^x) \gets \sum_{\omega_t\in\Omega_t} \Pb(\omega_t|I_{t-1}^x) \underline{\beta}_{t}^k(\omega_t).$$
Thus,
$$ h_{t-1}^k(R^{x}_{t-1}, I_{t-1}^x) :=  \alpha^k_{t-1}(I_{t-1}^x) + \langle\beta^k_{t-1}(I_{t-1}^x), R^x_{t-1} - R^{x,k}_{t-1}\rangle.$$
Hence, we can construct the new value function approximation $\overline{V}^k_{t-1}(R^x_{t-1}, I^x_{t-1}) $ for the post--decision information state $ I_{t-1}^x$ as,
\begin{equation}
\label{eq:vfam}
\overline{V}_{t-1}^{k}(R^x_{t-1}, I_{t-1}^x) := \max\big\{\overline{V}_{t-1}^{k-1}(R^x_{t-1},I_{t-1}^x),~h_{t-1}^k(R^x_{t-1}, I_{t-1}^x)\big\}.
\end{equation}
The description of the method is given in Algorithm \ref{a:reg-sddp-markov}.
\begin{algorithm}
	\caption{Quadratic Regularization Method for Markov Models}
\footnotesize
\begin{algorithmic}[1]
\STATE Choose $Q_t\succeq 0, t=0,\dots,T$, and define the sequence $\{\varrho^k\}$.
\STATE Define $\overline{V}_{T}^k(R^x_T, I_T^x):= V^*_T(R^x_T, I_T^x),~ k=0,\dots, K,~I_T^x\in\Ic_T^x$.
\STATE Define $\overline{V}_{t}^0(R^x_t, I_t(\omega_t)):=-\infty,~\omega_t\in\Omega_t,~ t=0,\dots, T-1$.
\STATE $(R^{x,k}_{-1}, I_0) \gets S_0,~ k = 0,\dots, K$
\FOR{$k=0,\dots, K$}
\STATE Sample $\omega\in\Omega$ using the Markov stochastic process $\{W_t, t=1,\dots,T\}$.
\FOR{$t=0,\dots,T$}
\IF{$(k = 0)$}
\STATE
\begin{equation*}
 \mbox{Select } x^k_t \in 
 \argmin_{x_t\in\Xc_t(R^{x,k}_{t-1}, I_t(\omega))}~ \left\{C(S_t(\omega), x_t)
 \right\}
\end{equation*}

\ELSE
\IF{$t < T$}
\STATE

\begin{equation*}
 x^k_t \in 
 \argmin_{x_t\in\Xc_t(R^{x,k}_{t-1}, I_t(\omega))}~\left\{ C(S_t(\omega), x_t) + \overline{V}^{k-1}_{t}(R^x_t, I^x_t(\omega)) +
      \frac{\varrho^k}{2} \Big\langle R^x_t - \overline{R}^{x,k-1}_t, Q_t(R^x_t - \overline{R}^{x,k-1}_t)\Big\rangle 
 \right\}
 \end{equation*}


\ELSE
\STATE

\begin{equation*}
 x^k_t \in 
 \argmin_{x_t\in\Xc_t(R^{x,k}_{t-1}, I_t(\omega))}~ \left\{C(S_t(\omega), x_t) + \overline{V}^{k-1}_{t}(R^x_t, I^x_t(\omega)) 
 \right\}
 \end{equation*}


\ENDIF
\ENDIF
\STATE $\mbox{Set } R_{t}^{x,k} \gets B_t^k x_t^k $; $S_{t+1}(\omega) \gets (R_t^{x,k} - b_{t+1}(\omega), I_{t+1}(\omega))$

\ENDFOR
\FOR{$t=T,\dots, 1$}
\STATE \begin{equation*}
  \mbox{Define } \underline{V}^k_t(R^{x}_{t-1}, \omega_t) := 
  \min_{x_t\in\Xc_t(R^{x}_{t-1}, I_t(\omega_t))}~ \left\{C(S_t(\omega_t), x_t) + \overline{V}^k_{t}(R^x_t, I^x_t(\omega_t))
  \right\} 
  \end{equation*}
  \FORALL{$\omega_{t}\in\Omega_{t}$}
\STATE \begin{equation*}
  \mbox{Select }\underline{\beta}^k_t(\omega_t) \in \partial_{R^x_{t-1}} \underline{V}^k_t(R_{t-1}^{x,k}, \omega_t) 
  \end{equation*}
  \ENDFOR
  \FORALL{$I_{t-1}^x\in \Ic_{t-1}^x(\Omega_{t-1})$}
  \STATE $\displaystyle \alpha^k_{t-1}(I_{t-1}^x)\gets \sum_{\omega_t\in\Omega_t} \Pb(\omega_t|I_{t-1}^x)\underline{V}^k_t(R^{x,k}_t, \omega_t)$; $\displaystyle \beta^k_{t-1}(I_{t-1}^x) \gets \sum_{\omega_t\in\Omega_t} \Pb(\omega_t|I_{t-1}^x) \underline{\beta}_{t}^k(\omega_t)$
  \STATE $\displaystyle h_{t-1}^k(R^{x}_{t-1}, I_{t-1}^x) :=  \alpha^k_{t-1}(I_{t-1}^x) + \langle\beta^k_{t-1}(I_{t-1}^x), R^x_{t-1} - R^{x,k}_{t-1}\rangle$
  \STATE $\displaystyle \overline{V}_{t-1}^{k}(R^x_{t-1}, I_{t-1}^x) := \max\big\{\overline{V}_{t-1}^{k-1}(R^x_{t-1},I_{t-1}^x),~h_{t-1}^k(R^x_{t-1}, I_{t-1}^x)\big\}$
  \ENDFOR
\ENDFOR
\STATE
  \begin{fleqn} \begin{equation*}
  \underline{V}^k_0 \gets \left\{
  \begin{split}
  \min_{x_0\in\Xc_0(S_0)}~&C(S_0, x_0) + \overline{V}^k_{0}(R^x_0, I^x_0)\\
  \end{split}
  \right\} 
  \end{equation*}
  \end{fleqn}
\STATE $\overline{R}^{x,k}_t \gets R^{x,k}_{t},~t=0,\dots, T-1$
\ENDFOR
\end{algorithmic}
\label{a:reg-sddp-markov}
\label{a:2}
\end{algorithm}
\normalsize

\begin{theorem} Suppose that $\{W_t, t=1,\dots,T\}$ is a discrete Markov process as described by equation (\ref{eq:14}). If  $\overline{V}_T^k(\cdot, I_T^x) \equiv V^*_T(\cdot, I_T^x),~\mbox{for }I_T^x\in\Ic^x_T,~k=0,\dots,K$, and  conditions 2, 3, 4, and 5 specified in Theorem \ref{t:7} are satisfied,
\label{t:9}
then the method presented in Algorithm \ref{a:2} converges to an optimal policy of problem (\ref{eq:2}) after a finite number of iterations
with probability one.
\end{theorem}
\begin{proof}{Proof:}
The proof is analogous to the proof of Theorem \ref{t:7}. The main difference is that for each time period $t=1,\dots,T$ we need to consider $|I^x_t|$ different value functions
$\overline{V}_t^k(R^x_t, I^x_t)$. Since $|I^x_t|$ is finite, the argument of the proof of Theorem \ref{t:7} can be extended to show that with probability 1, there exists a large enough $k\in\N$ such that the value functions
$\overline{V}_t^k(R^{\hat{x},k}_t(\omega), I^x_t)$ are optimal for all $\omega\in\Omega$, and $I^x_t \in \Ic^x_t, t=0,\dots,T$.
\end{proof}
\begin{remark}
Various optimization methods for Markov models have been studied in the literature for both the risk--neutral (see \cite{puterman2014markov, powell2007approximate, shapiro2014lectures}) and risk--averse cases
(see \cite{PhilpottDeMatos, ruszczynski2010risk, mo2001integrated} and the references within).
An extensive treatment of optimization problems with Markov uncertainty is beyond the scope of this work.
The goal of our presentation is the introduction of regularization into the Markovian setting, so that it can be adapted to other problems on a case--by--case basis.
\end{remark}

\section{Algorithmic Tuning}
\label{s:tuning}
In order to turn mathematical arguments into useful numerical results one needs to employ a high quality implementation and suitable parameter tuning.
In this section we present some of the potential issues regarding the reliability and computational performance of the methods presented above.
We consider the construction of regularization sequences, and discuss numerical concerns regarding the solutions of subproblems.
\subsection{Regularization Coefficients}
\label{s:reg_coeff}
In general, we cannot find a regularization sequence that would lead to the fastest possible convergence.
However, if we consider sequences that are defined by a set of parameters, then we can attempt to find suitable parameter values.
For example, we can construct regularization sequences $\varrho^k \geq 0$ such that $\displaystyle \lim_{k\rightarrow\infty} \varrho^k= 0$ by using the following geometric sequence.  Given $\varrho^0>0$ and $r\in(0,1)$, we define
\begin{equation}
\varrho^k = \varrho^0 r^{k} = r \cdot \varrho^{k-1} \mbox{, if } k > 0.
\end{equation}
In this case, we need to tune the parameters $\varrho^0$ and $r$. We can gain insight by solving a small instance of the given problem for different pairs $(\varrho^0, r)$.
For example, in section \ref{s:7} we describe an optimization model to be solved for high--dimensional post--decision resource states $|R^x_t|\geq 50$.
As a pre--processing step, we can solve a smaller instance, e.g. $|R^x_t| = 25$, for each $(\varrho^0,r)\in \{1,10,100\}\times\{0.9,0.95,0.99\}$,
and compare the results. Since estimates of the upper bounds and optimality gaps are stochastic, we prefer to compare only the deterministic lower bounds as they are more reliable.
The resulting plots can be found in Figure \ref{fig:1}. We can see that the sequences of regularization coefficients has an impact on the behavior of the proposed methods.
However, various choices of $(\varrho^0, r)$ can be used with similar success. In our experiments in section \ref{s:7}, we use $\varrho^0 = 1,~r=0.95$.
\input{test25x10-combined3.tex}

\subsection{Convergence Tolerance for the Solution of Subproblems}
At each step of the forward and backward pass of Algorithm \ref{a:reg-sddp} and Algorithm \ref{a:reg-sddp-markov},
we use the current collection of hyperplanes $\big\{\alpha^j_{t} + \langle\beta^j_{t}, R^x_t - R^{x,j}_t\rangle,~ j\leq k\big\}$ and a realization of $W_t = (A_t, B_t, b_t, c_t)$ as a part of the input to a convex optimization problems having the following general form,
\begin{equation}
\label{eq:21}
\begin{split}
\min~ &\langle c,y\rangle + \frac{1}{2}\langle y,Qy\rangle\\
\mbox{s.t. } &Ay=b\\
&y\geq 0
\end{split}
\end{equation}
The numerical precision of the solutions to subproblems (\ref{eq:21}) is essential for the correctness of the resulting policy for problem (\ref{eq:1}).
However, the right--hand side vector $b$ of problem (\ref{eq:21}) includes the vector $b_t$ and the constant terms $\alpha_t^j - \beta^j_t R^{x,j}_t$ of the value function approximations given in (\ref{eq:vfa}) or (\ref{eq:vfam}).
If problem (\ref{eq:1}) has a long time horizon, then an aggregation of constant terms with large modulus $|\alpha_t^j  - \beta^j_t R^{x,j}_t|$ can occur,
and that could lead to numerical solutions of problem (\ref{eq:21}) which do not satisfy the system of constraints $B_{t-1}x_{t-1} + A_t x_t = b_t$ with a desirable precision.
Convex optimization tools, including specialized algorithms for linear and quadratic programming problems, often use convergence tolerance parameters to guide their stopping conditions.
For problems with long time horizons, we encounter numerical precision problems that require that we use care in setting tolerance parameters for stopping conditions.  In the sections below, we discuss the issues of relative primal feasibility, and the relative complementarity gap.
\subsubsection{Relative primal feasibility}
Suppose that a given optimization solver has a feasibility condition of the following form,
\begin{equation}
\frac{||Ay - b||}{1+||b||} \leq \varepsilon_f.
\end{equation}
We can consider two right--hand side vectors $b^1, b^2$ such that $||b^1|| < ||b^2||$ and corresponding candidate solutions $y^1,y^2$ such that
$\displaystyle\frac{||Ay^1 - b^1||}{1+||b^1||} = \frac{||Ay^2 - b^2||}{1+||b^2||} = \varepsilon_f$.\\
Then the feasibility errors satisfy $\displaystyle||Ay^1 - b^1|| < ||Ay^2 - b^2||$.
Therefore, if we keep the primal feasibility tolerance $\varepsilon_f$ fixed while $||b||$ grows, then the feasibility errors $||Ay - b||$ (and therefore $||B_{t-1}x_{t-1} + A_tx_t - b_t||$) could increase as well.
Hence, for problems with a long time horizon or a large number of hyperplanes in the value function approximation,
one might need to decrease the tolerance $\varepsilon_f$ for problem (\ref{eq:21}) in order to bring the size of the error $||B_{t-1}x_{t-1} + A_tx_t - b_t||$ down to an acceptable level.

\subsubsection{Relative complementarity gap}
Commercial solvers often include implementations of primal--dual interior point methods (see \cite{wright1997primal, benson2009interior}) that employ a relative complementarity tolerance $\varepsilon_c$ in their stopping condition.
The presence of large (by modulus) constant terms in the right--hand side vector $b$ can also lead the numerical solver to terminate at an infeasible solution with non--negligible errors $||Ay-b||$ and
$||B_{t-1}x_{t-1} + A_tx_t - b_t||$,
if $\varepsilon_c$ is not chosen appropriately. We present a brief explanation below.\\
The Lagrangian of problem (\ref{eq:21}) is given by
\begin{equation}
L(y,\mu,\lambda) = \langle c,y\rangle + \frac{1}{2}\langle y,Qy\rangle + \langle \mu, b-Ay\rangle - \langle\lambda,y\rangle.
\end{equation}
Hence, the Karush--Kuhn--Tucker conditions for problem (\ref{eq:21}) are given by the system of constraints,
\begin{equation}
\label{eq:25}
\begin{split}
Ay &= b\\
A^{\top}\mu - Qy + \lambda &= c\\
Y\Lambda\mathbf{1} &= 0\\
y,\lambda &\geq 0
\end{split}
\end{equation}
where $Y = \mbox{diag}(y)$ and $\Lambda = \mbox{diag}(\lambda)$.

Interior point methods construct iterative approximations to the solution of (\ref{eq:25})  using a sequence of scalar barrier parameters
$\nu^n > 0$, such that $\nu^n\downarrow 0$.
Assuming that the initial point $(y^0,\mu^0, \lambda^0)$ is infeasible for (\ref{eq:25}) and $\langle  y^0, \lambda^0 \rangle > 0$, we can have a stopping condition for the complementarity gap
such as
\begin{equation}
\label{eq:26}
\frac{\langle y^n, \lambda^n\rangle }{\langle y^0, \lambda^0 \rangle} \leq \varepsilon_c ~~\mbox{  or  }~~ \nu^n \leq \varepsilon_c ~~\mbox{  or  }~~\displaystyle \frac{\nu^n}{|\langle c,y^n\rangle + \langle y^n,Qy^n\rangle|} \leq \varepsilon_c.
\end{equation}
At iteration $n$, the interior point method finds a Newton direction $(\Delta y,\Delta \mu, \Delta \lambda)$ for problem (\ref{eq:25}) as the solution to the following system :
\begin{equation}
\begin{bmatrix}
A & 0 & 0\\
-Q & A^\top & I\\
\Lambda & 0 & Y
\end{bmatrix}\cdot
\left[\begin{array}{c}
\Delta y\\
\Delta \mu\\
\Delta \lambda
\end{array}\right] = \left[
\begin{array}{c}
b-Ay^n\\
c-A^{\top}\mu^n + Qy^n + \lambda^n\\
\nu^n\mathbf{1} - Y^n\Lambda^n\mathbf{1}
\end{array}
\right]
\end{equation}
where $I=\mbox{diag}(1,1,\dots,1)$ denotes the identity matrix.\\
Then the current solution $(y^n, \mu^n,\lambda^n)$ can be updated by choosing $\gamma^n \in (0, 1]$ such that
\begin{equation}
(y^n, \lambda^n) + \gamma^n(\Delta y, \Delta\lambda)\geq 0
\end{equation}
and setting
\begin{equation}
(y^{n+1}, \mu^{n+1}, \lambda^{n+1}) = (y^{n}, \mu^{n}, \lambda^{n}) + \gamma^n (\Delta y, \Delta\mu, \Delta\lambda).
\end{equation}
Please note that if $\gamma^n = 1$, then $y^{n+1}\geq 0$ would be feasible for problem (\ref{eq:21}) since $Ay^{n+1} = b$.
However, in practice we usually have $\gamma^n < 1$.
Hence, a complementarity tolerance condition (\ref{eq:26}) can be met even if the system $B_{t-1}x_{t-1} + A_tx_t = b_t$ is not satisfied within the desired precision.
In order to address this concern, in our numerical experiments we set the tolerance $\varepsilon_c$ to the smallest possible value allowed by the solver ($10^{-12}$).

\section{Numerical Experiments}
\label{s:7}
In this section we study the computational performance of the algorithms proposed above.
We focus our analysis on the following questions.
\begin{itemize}
\item How is the computational performance of Algorithms \ref{a:reg-sddp} and \ref{a:reg-sddp-markov} affected by:
\begin{itemize}
\item the dimension of the resource vector $R_t$?
\item the size of the post--decision information state space ${\cal I}^x_t$?
\end{itemize}
\item How does the performance of  Algorithms \ref{a:reg-sddp} and \ref{a:reg-sddp-markov} compare to their non--regularized counterparts?
\end{itemize}

Our experimental work was conducted using the setting of optimizing grid level storage for a large transmission grid managed by PJM Interconnection.  PJM manages grid level storage devices from a single location, making it a natural setting for testing our algorithms.  As of this writing, grid level storage is dropping in price, providing a meaningful setting to evaluate the performance of our algorithms for a wide range of storage devices, challenging the ability of the algorithms to handle high dimensional applications.  For this reason, we conducted tests on networks with 50 to 500 storage devices.  These are much higher dimensional problems than prior research that has focused on the management of water reservoirs.  

Another distinguishing feature of our grid storage setting (compared to prior experimental work) is that a natural time step is 5 minutes, which is the frequency with which real--time electricity prices (known as LMPs, for locational marginal prices) are updated on the PJM grid.  We anticipate using storage devices to hold energy over horizons of several hours.  For this reason, we used a 24 hour model, divided into 5--minute increments, for 288 time periods, which is quite large compared to many applications using this algorithmic technology.

A complete description of the given model is beyond the scope of the current paper and can be found in \cite{asamov2015multistage}.
Below we briefly describe the construction of the network, and the exogenous stochastic process. Finally we present the results of an extensive set of experiments investigating the effect of regularization, the number of storage devices (which determines the dimensionality of $R_t$), and the presence of an exogenous post--decision information state, on the rate of convergence and solution quality. 

\subsection{The network}

We performed our experiments using an aggregated version of the PJM grid. Instead of the full network with 9,000 buses and 14,000 transmission lines, we limited our analysis to the higher voltage lines, producing a grid with 1,360 buses and 1,715 transmission lines.  The power generators include 396 gas turbines (23,309 MW), 50 combined cycle generators (21,248 MW),
264 steam generators (73,374 MW), 31 nuclear reactors (31, 086 MW), and 84 conventional hydro power generators (2,217 MW). Off--shore wind power was simulated for a set of hypothetical wind turbines with a combined maximum capacity of 16 GW.  Moreover, we consider a daily time horizon with 5--minute discretization resulting in a total of 288 time periods.

The data was prepared by first running a unit--commitment simulator called SMART--ISO that determines which generators are on or off at each point in time, given forecasts of wind generated from a planned set of off--shore wind farms.  We made the assumption that the use of grid level storage would not change which generators are on or off at any point in time.  However, we simultaneously optimize ramping the generators up or down within ranges, while charging and discharging of storage devices around the grid in the presence of stochastic injections from the wind farms.

We placed the distributed storage devices at the points--of--interconnection for wind farms, as well as the buses with the highest demand. Each storage device is characterized by its minimum and maximum energy capacity, its charging and discharging efficiency, and its variable storage cost. The control of multiple storage devices in a distributed energy system is a challenging task that depends on a variety of factors such as the location of each device, and the presence of transmission line
congestion. A good storage algorithm needs to respond to daily variations in supply, demand and congestion, taking advantage of opportunities to store energy near generation points (to avoid congestion) or near load points (during off--peak periods). It has to balance when and where to store and discharge in a stochastic, time--dependent setting, providing a challenging test environment for our algorithm.

\subsection{The exogenous information}
\label{s:7.1}

Our only source of uncertainty (the exogenous information) was from the injected wind from the offshore wind farms.  In order to calibrate our stochastic wind error model, we employed historical wind data and speed measurements of off--shore wind for the month of January 2010.  For each time period, we consider a set of ten vectors of possible wind speed realizations which correspond to ten different weather regimes.

In general, the exogenous information process can be characterized by one of the following: stagewise independence, compact state variables (Markov processes), or scenario--dependence (path dependence). For some instances, the latter case could be reduced to one of the former two by applying an appropriate transformation as described in section \ref{s:5}.
In our experiments, we consider instances with stagewise independent transitions between ten equally likely scenarios.  When we assumed stagewise independence, we would sample from each of these 10 scenarios with equal probability at each time period.  For the problems with Markov uncertainty, we assumed that at every time period $t$, the probability of continuing with the same weather regime at time $t+1$ is 91 percent.  Additionally, each of the remaining nine regimes can be visited at time $t+1$ with a probability of 1 percent.
\begin{figure}
\begin{tikzpicture}
\pgfplotsset{every axis legend/.append style={font=\footnotesize,at={(0.45,-0.25)},anchor=north, legend columns=5}}
\begin{axis}[
height=2.1in,
width=4.8in,
title={Wind Power Simulations},
grid=major,
date coordinates in=x,
xticklabel={\hour:\minute},
x tick label style={align=center},
date ZERO=2010-01-11, 
xmin={2010-01-11 12:00}, 
xmax={2010-01-12 12:00},
xtick={2010-01-11 12:00, 
2010-01-11 18:00,
2010-01-11 24:00,
2010-01-12 6:00,
2010-01-12 12:00},
xlabel={Time},
ylabel={MWh}
]
\addplot[color=red, line width = 1] table [col sep=comma,trim cells=true,y=$S_1$] {data_sim_only.csv}; \addlegendentry{Simulation 1}
\addplot[color=blue, line width = 1] table [col sep=comma,trim cells=true,y=$S_2$] {data_sim_only.csv}; \addlegendentry{Simulation 2}
\addplot[color=green, line width = 1] table [col sep=comma,trim cells=true,y=$S_3$] {data_sim_only.csv}; \addlegendentry{Simulation 3}
\addplot[color=black, line width = 1] table [col sep=comma,trim cells=true,y=$S_4$] {data_sim_only.csv}; \addlegendentry{Simulation 4}
\addplot[color=orange, line width = 1] table [col sep=comma,trim cells=true,y=$S_5$] {data_sim_only.csv}; \addlegendentry{Simulation 5}
\addplot[color=purple, line width = 1] table [col sep=comma,trim cells=true,y=$S_6$] {data_sim_only.csv}; \addlegendentry{Simulation 6}
\addplot[color=brown, line width = 1] table [col sep=comma,trim cells=true,y=$S_7$] {data_sim_only.csv}; \addlegendentry{Simulation 7}
\addplot[color=gray, line width = 1] table [col sep=comma,trim cells=true,y=$S_8$] {data_sim_only.csv}; \addlegendentry{Simulation 8}
\addplot[color=yellow, line width = 1] table [col sep=comma,trim cells=true,y=$S_9$] {data_sim_only.csv}; \addlegendentry{Simulation 9}
\addplot[color=cyan, line width = 1] table [col sep=comma,trim cells=true,y=$S_{10}$] {data_sim_only.csv}; \addlegendentry{Simulation 10}

\end{axis}
\end{tikzpicture}
\caption{Simulated daily realizations of wind power for a given wind farm over 24 hour time horizon.}
\label{fig:wind}
\end{figure}
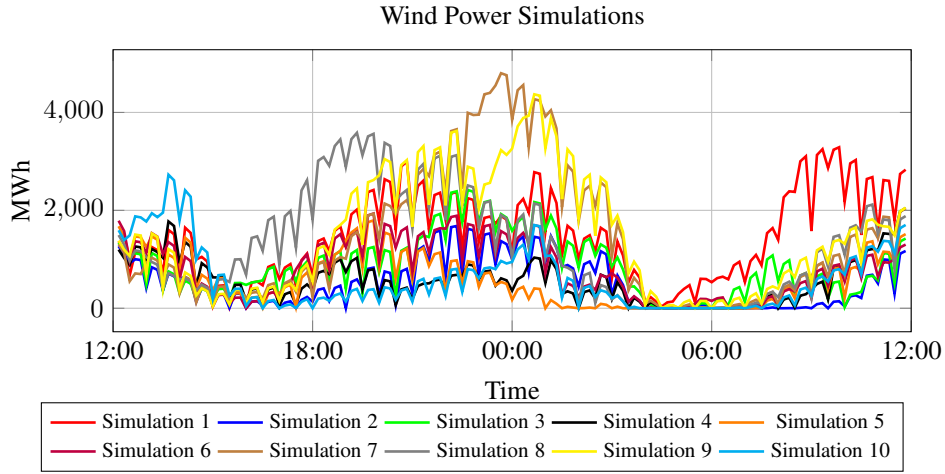
 
\subsection{Algorithmic comparisons}

The proposed algorithms were implemented in Java, and the IBM ILOG CPLEX 12.4  solver was used for the solution of linear and quadratic convex optimization problems. Further, we performed parameter tuning as described in section \ref{s:tuning}. We set the relative complementarity tolerance of CPLEX to $10^{-12}$, and used a geometric regularization sequence with $\varrho^0 = 1$ and $r=0.95$. Additionally, we run each method for $K = 300$ iterations. 
Moreover, the scaling matrices $Q_t, t=0,\dots,T$ are set to the identity matrix which implies that the amount of energy in each storage device has the same weight in the regularization term. In this section we examine the performance of Algorithms \ref{a:reg-sddp} and \ref{a:reg-sddp-markov} when the number of storage devices (dimension of the resource state variable) is $|R^x_t| = 50, 100, 200, 500$.

Plots of the behavior of the methods can be found in Figures \ref{fig:50}, \ref{fig:100}, \ref{fig:200}, \ref{fig:500} below. Each figure shows the results for stagewise independence on the left, and Markov uncertainty on the right. These graphs show the convergence of the upper and lower bounds, illustrating the dramatic impact of regularization, especially as the number of dimensions grow.  The results suggest that we consistently obtain high quality solutions within approximately 50 iterations for all problems.
\input{test50x10-combined1.tex}
\input{test100x10-combined1.tex}
\input{test200x10-combined1.tex}
\input{test500x10-combined2.tex}

Table \ref{t:swi} and Table \ref{t:m} show the CPU times (in seconds) per iteration for problems with 50 to 500 storage devices, with stagewise independence and Markov uncertainty, for up to 300 iterations.  We note that in a practical application, the algorithms would be run offline (for example, the day before, given a particular forecast of wind).  The cuts would be stored and then used in real time the next day.  This would be easily implementable in a policy updated every 5 minutes.  
\begin{table}
\begin{center}
\begin{tabular}{|l|c|c|c|c|c|c|c|c|}
\hline
\multicolumn{2}{|c|}{\backslashbox{$|R^x_t|$}{\# Iterations}} & 1 & 50 & 100 & 150 & 200 & 250 & 300 \\
\hline
50 
& \begin{tabular}{@{}c@{}}Algorithm 1\\SDDP\end{tabular} 
&  \begin{tabular}{@{}c@{}}182.0\\ 181.0\end{tabular}
&  \begin{tabular}{@{}c@{}}217.6\\ 196.2\end{tabular}
&  \begin{tabular}{@{}c@{}}230.8\\ 201.2\end{tabular}
&  \begin{tabular}{@{}c@{}}248.3\\ 208.9\end{tabular}
&  \begin{tabular}{@{}c@{}}266.5\\ 215.4\end{tabular}
&  \begin{tabular}{@{}c@{}}284.3\\ 223.4\end{tabular}
&  \begin{tabular}{@{}c@{}}299.2\\ 230.4\end{tabular}
\\ \hline
100 
& \begin{tabular}{@{}c@{}}Algorithm 1\\SDDP\end{tabular} 
&  \begin{tabular}{@{}c@{}}237.0\\ 246.0\end{tabular}
&  \begin{tabular}{@{}c@{}}306.2\\ 250.0\end{tabular}
&  \begin{tabular}{@{}c@{}}334.3\\ 262.2\end{tabular}
&  \begin{tabular}{@{}c@{}}371.7\\ 275.1\end{tabular}
&  \begin{tabular}{@{}c@{}}412.9\\ 296.0\end{tabular}
&  \begin{tabular}{@{}c@{}}453.9\\ 319.7\end{tabular}
&  \begin{tabular}{@{}c@{}}500.4\\ 341.4\end{tabular}
\\ \hline
200 
& \begin{tabular}{@{}c@{}}Algorithm 1\\SDDP\end{tabular} 
&  \begin{tabular}{@{}c@{}}293.0\\ 265.0\end{tabular}
&  \begin{tabular}{@{}c@{}}358.8\\ 375.3\end{tabular}
&  \begin{tabular}{@{}c@{}}414.3\\ 360.7\end{tabular}
&  \begin{tabular}{@{}c@{}}507.0\\ 394.5\end{tabular}
&  \begin{tabular}{@{}c@{}}587.6\\ 428.8\end{tabular}
&  \begin{tabular}{@{}c@{}}653.5\\ 469.8\end{tabular}
&  \begin{tabular}{@{}c@{}}726.0\\ 525.5\end{tabular}
\\ \hline
500 
& \begin{tabular}{@{}c@{}}Algorithm 1\\SDDP\end{tabular} 
&  \begin{tabular}{@{}c@{}}553.0\\ 332.0\end{tabular}
&  \begin{tabular}{@{}c@{}}664.0\\ 426.5\end{tabular}
&  \begin{tabular}{@{}c@{}}828.4\\ 564.6\end{tabular}
&  \begin{tabular}{@{}c@{}}995.4\\ 651.5\end{tabular}
&  \begin{tabular}{@{}c@{}}1183.3\\ 751.6\end{tabular}
&  \begin{tabular}{@{}c@{}}1673.5\\ 869.8\end{tabular}
&  \begin{tabular}{@{}c@{}}2536.0\\ 1003.2\end{tabular}
\\ \hline
\end{tabular}
\caption{Computational time per iteration (in seconds) for risk--neutral stochastic optimization methods.}
\label{t:swi}
\end{center}
\end{table}
\begin{table}
\begin{center}
\begin{tabular}{|l|c|c|c|c|c|c|c|c|}
\hline
\multicolumn{2}{|c|}{\backslashbox{$|R^x_t|$}{\# Iterations}} & 1 & 50 & 100 & 150 & 200 & 250 & 300 \\
\hline
50 
& \begin{tabular}{@{}c@{}}Algorithm 2\\MSDDP\end{tabular} 
&  \begin{tabular}{@{}c@{}}180.0\\ 181.0\end{tabular}
&  \begin{tabular}{@{}c@{}}225.6\\ 192.7\end{tabular}
&  \begin{tabular}{@{}c@{}}239.1\\ 198.2\end{tabular}
&  \begin{tabular}{@{}c@{}}258.2\\ 206.3\end{tabular}
&  \begin{tabular}{@{}c@{}}277.0\\ 213.3\end{tabular}
&  \begin{tabular}{@{}c@{}}294.2\\ 221.8\end{tabular}
&  \begin{tabular}{@{}c@{}}310.1\\ 228.9\end{tabular}
\\ \hline
100 
& \begin{tabular}{@{}c@{}}Algorithm 2\\MSDDP\end{tabular} 
&  \begin{tabular}{@{}c@{}}256.0\\ 255.0\end{tabular}
&  \begin{tabular}{@{}c@{}}309.5\\ 255.7\end{tabular}
&  \begin{tabular}{@{}c@{}}336.1\\ 267.2\end{tabular}
&  \begin{tabular}{@{}c@{}}371.5\\ 279.3\end{tabular}
&  \begin{tabular}{@{}c@{}}411.2\\ 300.8\end{tabular}
&  \begin{tabular}{@{}c@{}}450.3\\ 325.1\end{tabular}
&  \begin{tabular}{@{}c@{}}495.6\\ 347.1\end{tabular}
\\ \hline
200 
& \begin{tabular}{@{}c@{}}Algorithm 2\\MSDDP\end{tabular} 
&  \begin{tabular}{@{}c@{}}296.0\\ 339.0\end{tabular}
&  \begin{tabular}{@{}c@{}}364.6\\ 301.6\end{tabular}
&  \begin{tabular}{@{}c@{}}422.2\\ 319.4\end{tabular}
&  \begin{tabular}{@{}c@{}}513.0\\ 364.9\end{tabular}
&  \begin{tabular}{@{}c@{}}592.6\\ 409.3\end{tabular}
&  \begin{tabular}{@{}c@{}}657.4\\ 454.8\end{tabular}
&  \begin{tabular}{@{}c@{}}731.1\\ 515.1\end{tabular}
\\ \hline
500 
& \begin{tabular}{@{}c@{}}Algorithm 2\\MSDDP\end{tabular} 
&  \begin{tabular}{@{}c@{}}542.0\\ 338.0\end{tabular}
&  \begin{tabular}{@{}c@{}}650.3\\ 434.7\end{tabular}
&  \begin{tabular}{@{}c@{}}799.5\\ 586.7\end{tabular}
&  \begin{tabular}{@{}c@{}}959.4\\ 674.3\end{tabular}
&  \begin{tabular}{@{}c@{}}1151.7\\ 777.2\end{tabular}
&  \begin{tabular}{@{}c@{}}1637.6\\ 886.8\end{tabular}
&  \begin{tabular}{@{}c@{}}2490.4\\ 1004.1\end{tabular}
\\ \hline
\end{tabular}
\caption{Computational time per iteration (in seconds) for risk--neutral stochastic optimization methods.}
\label{t:m}
\end{center}
\end{table}

\section{Conclusion}
\label{s:8}
Large scale multistage stochastic optimization problems with long time horizons arise in numerous real--world applications in energy, finance, transportation and other fields.
The numerical solution of such models can be computationally demanding, often causing practioners to face a trade--off between solution quality and computational time.

In our work we have developed regularization methods for the SDDP framework and studied their convergence.
The algorithms employ regularization terms in the selection of cutting hyperplanes which improve the quality of the resulting value function approximations.
The proposed techniques feature straightforward implementation and can be quickly integrated into existing software solutions without the need for major additional efforts in development and testing.

In order to assess the performance of the proposed approach we consider a model for the integration of renewable energy using distributed grid--level storage into the grid of PJM, one of the largest regional transmission operators in the United States.
Our numerical experiments indicate that the proposed regularized algorithms exhibits significantly faster convergence than their non--regularized counterparts, with greater gains observed for higher--dimensional problems.

In the future we can consider several extensions of the current work. One possible direction would involve further investigation into the selection of appropriate regularization terms and coefficients.
Another possible path of exploration would be the application of regularization techniques for the solution of risk--averse models involving time--consistent compositions of coherent measures of risk along the lines of \cite{kozmik2013risk, asamov2014risk}, and \cite{shapiro2013risk}.
Additionally, we would also like to extend the proposed approach to the solution of multiobjective stochastic models \cite{young2010multiobjective}.
Finally, obtaining further empirical results and insights from problems in the field would also be a subject of great interest.


\bibliographystyle{siamplain}
\bibliography{myrefs}
\end{document}